\PassOptionsToPackage{svgnames}{xcolor}

\documentclass[a4paper,11pt,fleqn,reqno,oneside]{amsart}

\usepackage{lmodern}

\usepackage{mathtools}

\usepackage{amssymb}
\usepackage{amsthm}
\usepackage{amsfonts}

\usepackage{mathrsfs} 
\usepackage[english]{babel}
\usepackage{upgreek}
\usepackage{esvect}
\usepackage{extarrows}
\usepackage[noadjust]{cite}

\usepackage[foot]{amsaddr}

\let\Gamma\varGamma

\usepackage{xcolor}

\usepackage{todonotes}

\usepackage[
	paper = a4paper,
	top = 32mm, bottom = 32mm,
	left = 34mm, right = 34mm,
	footskip = 10mm, footnotesep=5mm, 
]
{geometry}

\usepackage{doi}

\usepackage{hyperref}
\hypersetup{
  pdftitle    = {A Bernstein Theorem for Minimal Maps with small Second Fundamental Form},
  pdfsubject  = {},
  pdfauthor   = {Felix Lubbe},
  pdfkeywords = {Minimal submanifolds, Bernstein theorem, ...},
  pdfcreator  = {},
  pdfproducer = {LaTeX with hyperref},
  hidelinks, 
  colorlinks,
  linkcolor=blue,
  citecolor=blue,
  urlcolor=RoyalBlue
}

\usepackage[sorted-cites]{amsrefs}

\usepackage{etoolbox}
\patchcmd{\section}{\scshape}{\bfseries}{}{}
\makeatletter
\renewcommand{\@secnumfont}{\bfseries}
\patchcmd{\@startsection}
  {\@afterindenttrue}
  {\@afterindentfalse}
  {}{}
\makeatother

\usepackage{enumitem}

\newcommand{\Rb}{\mathbb{R}}

\newcommand{\Lscr}{\mathscr{L}}

\newcommand{\Sym}{\mathrm{Sym}}

\newcommand{\A}{\mathrm{A}}
\newcommand{\Hv}{\vv{H}}

\DeclareMathOperator{\tr}{\mathrm{tr}}

\DeclareMathOperator*{\rank}{\mathrm{rank}}

\newcommand{\gMetric}{\mathrm{g}}

\DeclareMathOperator{\dd}{d\hspace{-3pt}}
\DeclareMathOperator{\Ric}{Ric}

\newcommand{\RMN}{\mathrm{R}_{M\times N}}
\newcommand{\RM}{\mathrm{R}_M}
\newcommand{\RN}{\mathrm{R}_N}
\newcommand{\Rind}{\mathrm{R}}

\newcommand{\gM}{\gMetric_M}
\newcommand{\gN}{\gMetric_N}
\newcommand{\gMN}{\gMetric_{M\times N}}
\newcommand{\gind}{\gMetric}

\newcommand{\sTensor}{\mathrm{s}}
\newcommand{\sMN}{\sTensor_{M\times N}}
\newcommand{\sind}{\sTensor}

\newcommand{\svdf}{\lambda}			% singular values of \dd f

\newtheorem{mainthm}{Theorem}
\newtheorem{maincor}[mainthm]{Corollary}

\newtheorem{theorem}{Theorem}
\newtheorem{lemma}[theorem]{Lemma}
\newtheorem{corollary}[theorem]{Corollary}
\newtheorem{remark}[theorem]{Remark}
\newtheorem*{remark*}{Remark}
\newtheorem{definition}[theorem]{Definition}
\newtheorem*{definition*}{Definition}
\newtheorem{example}[theorem]{Example}

\numberwithin{equation}{section}
\numberwithin{theorem}{section}

\allowdisplaybreaks

\title[Minimal Maps with Small Second Fundamental Form]{A Bernstein Theorem for Minimal Maps\\with small Second Fundamental Form}
\author{Felix Lubbe}
\address{Department of Mathematics, University of Hamburg, Bundesstr.\ 55, D--20146 Hamburg, Germany}
\email{Felix.Lubbe@uni-hamburg.de}

\begin{document}

	\begin{abstract}
		We consider minimal maps $f:M\to N$ between Riemannian manifolds 
		$(M,\mathrm{g}_M)$ and $(N,\mathrm{g}_N)$, where $M$ is compact and 
		where the sectional curvatures satisfy $\sec_N\le \sigma\le \sec_M$ for 
		some $\sigma>0$. Under certain assumptions on the differential of the
		map and the second fundamental form of the graph $\Gamma(f)$ of $f$, we show 
		that $f$ is either the constant map or a totally geodesic isometric immersion.
	\end{abstract}

	\maketitle
	\setcounter{tocdepth}{1}
	\tableofcontents
	
	%%%%%%%%%%%%%%%%%%%%%%%%%%%%%%%%%%%%%%%%%%%%%%%%%%%%%%%%%%%%%%%%%%%%%%%%%%%%%%%%%%
	% INTRODUCTION
	%%%%%%%%%%%%%%%%%%%%%%%%%%%%%%%%%%%%%%%%%%%%%%%%%%%%%%%%%%%%%%%%%%%%%%%%%%%%%%%%%%
	\section{Introduction}
	
	\setlength{\mathindent}{1.5cm} % !! some arbitrary value, needs to be fixed!
	
	We consider smooth minimal maps $f:M\to N$ between Riemannian
	manifolds $(M,\gM)$ and $(N,\gN)$. The map $f$
	is called \emph{minimal}, if its graph
	\[
		\Gamma(f) \coloneqq \bigl\{ (x,f(x)) : x\in M \bigr\} \qquad \subset M\times N
	\]
	is a minimal submanifold of $\bigl( M\times N, \gM\times\gN \bigr)$ \cite{Sch93}.
	The Bernstein theorem asserts that any complete minimal surface in $\Rb^3$ which can
	be written as a graph of a function on $\Rb^2$ must be a plane. This result
	has been generalized to $\Rb^n$ for $n\le 7$ and to general dimensions
	under various growth conditions (see e.\,g.\ \cite{EH90} and references therein). \\
	
	A generalized Bernstein problem is to ask under which additional geometric conditions
	the graph $\Gamma(f)$ is totally geodesic. Several results for higher codimension have
	been obtained, e.\,g.\ by assuming conditions on the slope of the graph
	and by considering its Gauss map \cite{FC80,HJW80,JX99,SWX06,JXY13,JXY16,Wang03}. \\
	
	Let us assume that $M$ is compact.
	A further approach to the problem is to demand a volume-decreasing condition for the map $f$.
	The map $f$ is called \emph{weakly length-decreasing} if
	\[
		\|\dd f(v)\|_{\gN} \le \|v\|_{\gM} \qquad \text{for all} \quad v\in\Gamma(TM)
	\]
	and \emph{weakly area-decreasing} if
	\[
		\|\dd f(v)\wedge \dd f(w)\|_{\gN} \le \|v\wedge w\|_{\gM} \qquad \text{for all} \quad v,w\in\Gamma(TM) \,.
	\]
	The length-decreasing property may be expressed in terms of the symmetric
	tensor
	\[
		\sind \coloneqq \gM - f^*\gN \,,
	\]
	i.\,e.\ $f$ is weakly length-decreasing precisely if $\sind$ is non-negative definite.
	This tensor satisfies a nice differential equation, 
	and allows the application
	of a tensorial maximum principle (see \cite{SHS13}, also for the area-decreasing case). \\
	
	Instead of considering the above tensor and its eigenvalues directly,
	we will study the behavior of minimal maps in relation to 
	the trace
	\[
		\tr(\sind) \coloneqq \sum_{k=1}^m \sind(e_k,e_k) \,,
	\]
	where $\{e_1,\dotsc,e_m\}$ is a local frame of $TM$, orthonormal with respect to the metric $\gind$
	induced by the embedding $\Gamma(f)\subset M\times N$. We remark
	that $\frac{1}{m}\tr(\sind)$ is the mean value of the eigenvalues of $\sind$,
	i.\,e.\ the trace also may be seen as an average.
	After these preparations, we are able to state the main result.
	
	\begin{mainthm}
		\label{thm:ThmA}
		Let $(M,\gM)$ and $(N,\gN)$ be Riemannian manifolds and
		suppose $M$ to be compact of dimension $\dim M \ge 2$.
		Assume that there exists $\sigma > 0$,  
		such that the sectional curvatures $\sec_{M}$ of $M$ 
		and $\sec_{N}$ of $N$ satisfy the relation
		\[
			\sec_{N} \le \sigma \le \sec_{M} \,.
		\]
		Let $f:M\to N$ be a minimal map and assume that the tensor $\sind$ satisfies
		the condition%
		\begin{equation}
			\label{eq:trCon}
			\tr(\sind) \ge 0 \,.
		\end{equation}
		Further, assume there is a finite constant $\kappa$, such that
		\[
			\kappa^{2} > 1 \qquad \text{and} \qquad f^{*}\gN < \kappa^{2} \gM
		\]
		and such that
		the second fundamental tensor $\A$ of the graph satisfies the bound
		\begin{equation}
			\label{eq:sCondition4}
			\|\A\|^{2} \le \frac{\kappa^{2} \sigma}{\kappa^{4}-1} \tr(\sind) \,.
		\end{equation}
		Then one of the following holds.
		\begin{enumerate}[label=(\roman*)]
			\item $f$ is the constant map.
			\item\label{it:thmAii} $f$ is a totally geodesic isometric immersion,
				$\sec_M=\sigma$
			    and the restriction of $\sec_N$
				to $\dd f(TM)$ is equal to $\sigma$. Moreover, $\Gamma(f)$ is a totally
				geodesic submanifold of $M\times N$.
		\end{enumerate}
	\end{mainthm}
		
	If the target space has at most half the dimension of the domain,
	the condition on the trace is automatically satisfied, i.\,e.\
	the following statement holds.
		
	\begin{maincor}
		\label{cor:CorB}
		Let $(M,\gM)$ be compact of
		dimension $m=\dim M\ge 2$, 
		let $(N,\gN)$ be of dimension at most $\frac{m}{2}$,
		and assume $\sec_N\le \sigma \le \sec_M$ for some $\sigma>0$. 
		If a minimal map $f:M\to N$
		satisfies the condition \eqref{eq:sCondition4},
		then $f$ is the constant map.
	\end{maincor}
		
	Let us shortly describe the structure of the paper.
	In section \ref{sec:SEMP}, we recall the strong elliptic maximum
	principle in vector bundles \cite{SHS13}, as required by the subsequent sections.
	Section \ref{sec:Maps} describes the geometry of graphs, and
	section \ref{sec:Proofs} provides the proof of theorem \ref{thm:ThmA}. 
	The paper ends with a short discussion of the assumptions 
	in section \ref{sec:Discussion}. \\
			
	\textbf{Acknowledgements.}
	I am grateful to
	Andreas Savas-Halilaj for 
	valuable discussions. This research was initiated while I was 
	supported 
	by the Research Training Group 1463 of the DFG 
	at the Leibniz Universit\"{a}t Hannover.

	%%%%%%%%%%%%%%%%%%%%%%%%%%%%%%%%%%%%%%%%%%%%%%%%%%%%%%%%%%%%%%%%%%%%%%%%%%%%%%%%%%
	% THE STRONG ELLIPTIC MAXIMUM PRINCIPLE
	%%%%%%%%%%%%%%%%%%%%%%%%%%%%%%%%%%%%%%%%%%%%%%%%%%%%%%%%%%%%%%%%%%%%%%%%%%%%%%%%%%
	\section{The Strong Elliptic Maximum Principle}
	\label{sec:SEMP}
	
	In \cite{SHS13},  Savas-Halilaj and Smoczyk provided a strong
	elliptic maximum principle for vector bundles, which we would
	like to apply. We briefly collect the results relevant for us.
	All manifolds will be smooth and connected without boundary. \\
	
	Let $(E,\pi,M)$ be a vector bundle of rank $k$ over a smooth
	manifold $M$. Suppose $\gMetric_E$ is a bundle metric on $E$
	and that $\nabla$ is a metric connection on $E$. A uniformly
	elliptic operator $\Lscr$ on $\Gamma(E)$ of second order is
	locally given by
	\[
		\Lscr = \sum_{i,j=1}^m a^{ij} \nabla_{e_i,e_j}^2 + \sum_{j=1}^m b^j \nabla_{e_j} \,,
	\]
	where $a\in\Gamma(TM\otimes TM)$ is a symmetric, uniformly positive definite
	tensor and $b\in\Gamma(TM)$ is a smooth vector field, such that
	\[
		a = \sum_{i,j=1}^m a^{ij} e_i\otimes e_j \qquad \text{and} \qquad b = \sum_{j=1}^m b^j e_j
	\]
	in a local frame field $\{e_1,\dotsc,e_m\}$ of $TM$. \\
	
	The following definition goes back to Hamilton \cite[Sec.\ 9]{Ham82}.
	
	\begin{definition}
		A fiberwise map $\Psi:\Sym(E^*\otimes E^*) \to \Sym(E^*\otimes E^*)$
		is said to satisfy the null-eigenvector condition, if whenever $\vartheta$
		is a non-negative symmetric $2$-tensor at a point $x\in M$ and if $v\in T_xM$
		is a null-eigenvector of $\vartheta$, then $\Psi(\vartheta)(v,v)\ge 0$.
	\end{definition}
	
	The elliptic analogue of Hamilton's maximum principle \cite{Ham86}*{Lemma 8.2}
	is given by the next theorem.
	
	\begin{theorem}[{\cite{SHS13}*{Theorem 2.3}}]
		\label{thm:SMP}
		Let $(M,\gM)$ be a Riemannian manifold and let $(E,\pi,M)$ be a Riemannian
		vector bundle over the manifold $M$ equipped with a metric connection.
		Suppose that $\phi\in\Sym(E^*\otimes E^*)$ is non-negative definite and
		satisfies
		\[
			\Lscr \phi + \Psi(\phi) = 0 \,,
		\]
		where $\Psi$ is a smooth fiberwise map satisfying the null-eigenvector condition.
		If there is an interior point of $M$ where $\phi$ has a null-eigenvalue, then
		$\phi$ must have a null-eigenvalue everywhere.
	\end{theorem}
	
	For $\phi\in\Sym(E^*\otimes E^*)$ a real number $\lambda$ is called
	\emph{eigenvalue} of $\phi$ with respect to $\gMetric_E$ at the point
	$x\in M$, if there  exists a non-zero vector $v\in E_x\coloneqq \pi^{-1}(x)$, such that
	\[
		\phi(v,w) = \lambda \gMetric_E(v,w) \,,
	\]
	for any $w\in E_x$. Since the tensor $\phi$ is symmetric, it admits
	$k$ real eigenvalues $\lambda_1(x),\dotsc,\lambda_k(x)$ at each point $x\in M$.
	We will always arrange the eigenvalues, such that
	\[
		\lambda_1(x)\le\dotsb\le \lambda_k(x)\,.
	\]
	
	The following result is also due to Hamilton \cite{Ham82} (see also \cite{SHS13}*{Theorem 2.4}).
		
	\begin{theorem}[Second Derivative Criterion]
		\label{thm:SDC}
		Suppose that $(M,\gM)$ is a Riemannian manifold
		and $(E,\pi,M)$ a Riemannian vector bundle of rank $k$ over $M$ equipped
		with a metric connection $\nabla$. Let $\phi\in\Sym(E^*\otimes E^*)$
		be a smooth symmetric $2$-tensor. If the largest eigenvalue $\lambda_k$
		of $\phi$ admits a local maximum $\lambda$ at an interior point $x_0\in M$, then
		\[
			(\nabla \phi)(v,v) = 0 \qquad \text{and} \qquad (\Lscr\phi)(v,v) \le 0 \,,
		\]
		for all vector $v$ in the eigenspace 
		\[
			\{ v \in E_{x_0} : \phi(v,w) = \lambda \gMetric_E(v,w) \;\; \text{for all} \;\; w\in E_{x_0} \}
		\]
		of $\lambda$ at $x_0$ and for all uniformly
		elliptic second order operators $\Lscr$.
	\end{theorem}
			
	%%%%%%%%%%%%%%%%%%%%%%%%%%%%%%%%%%%%%%%%%%%%%%%%%%%%%%%%%%%%%%%%%%%%%%%%%%%%%%%%%%
	% MAPS BETWEEN MANIFOLDS
	%%%%%%%%%%%%%%%%%%%%%%%%%%%%%%%%%%%%%%%%%%%%%%%%%%%%%%%%%%%%%%%%%%%%%%%%%%%%%%%%%%
	\section{Maps between Manifolds}
	\label{sec:Maps}
	
	%%%%%%%%%%%%%%%%%%%%%%%%%%%%%%%%%%%%%%%%%%%%%%%%%%%%%%%%%%%%%%%%%%%%%%%%%%%%%%%%%%
	% GEOMETRY OF GRAPHS
	%%%%%%%%%%%%%%%%%%%%%%%%%%%%%%%%%%%%%%%%%%%%%%%%%%%%%%%%%%%%%%%%%%%%%%%%%%%%%%%%%%
	\subsection{Geometry of Graphs}
	We recall the geometric quantities in a graphical
	setting, where we follow the presentation in \cite[Section 3.1]{SHS13}. \\

	Let $(M,\gM)$ and $(N,\gN)$ be Riemannian
	manifolds of dimensions $m$ and $n$, respectively. The induced
	metric on the product manifold will be denoted by
	\[
		\gMN = \gM\times \gN\,.
	\]
	A smooth map $f:M\to N$ defines an embedding $F:M\to M\times N$, by
	\[
		F(x) = \big(x,f(x)\big)\,,\qquad x\in M\,.
	\]
	The graph of $f$ is defined to be the submanifolds $\Gamma(f) \coloneqq F(M)$.
	Since $F$ is an embedding, it induces another Riemannian metric
	$\gMetric \coloneqq F^{*}\gMN$ on $M$. The two natural
	projections
	\[
		\pi_{M} : M\times N \to M\,,\qquad \pi_{N}:M\times N\to N\,,
	\]
	are submersions, that is they are smooth and have maximal rank. Note
	that the tangent bundle of the product manifold $M\times N$ splits
	as a direct sum
	\[
		T(M\times N) = TM \oplus TN\,.
	\]
	The four metrics $\gM,\gN,\gMN$ and $\gMetric$ 
	are related by
	\begin{align*}
		\gMN &= \pi_{M}^{*} \gM + \pi_{N}^{*}\gN \,, \\
		\gMetric &= F^{*}\gMN = \gM + f^{*}\gN \,.
	\end{align*}
	Now let us define the symmetric $2$-tensors
	\begin{align*}
		\sTensor_{M\times N} &\coloneqq \pi_{M}^{*}\gM - \pi_{N}^{*}\gN\,, \\
		\sTensor &\coloneqq F^{*}\sTensor_{M\times N} = \gM - f^{*}\gN \,.
	\end{align*}
	Note that $\sTensor_{M\times N}$ is a semi-Riemannian metric of signature $(m,n)$ on the
	manifold $M\times N$. \\
	
	The Levi-Civita connection $\nabla^{\gMN}$ associated to the
	Riemannian metric $\gMN$ on $M\times N$ is related to the Levi-Civita
	connections $\nabla^{\gM}$ on $(M,\gM)$ and $\nabla^{\gN}$
	on $(N,\gN)$ by
	\[
		\nabla^{\gMN} = \pi_{M}^{*}\nabla^{\gM} \oplus \pi_{N}^{*} \nabla^{\gN}\,.
	\]
	The corresponding curvature operator $\RMN$ on $M\times N$ with respect
	to the metric $\gMN$ is related to the curvature operators $\RM$ on
	$(M,\gM)$ and $\RN$ on $(N,\gN)$ by
	\[
		\RMN = \pi_{M}^{*}\RM \oplus \pi_{N}^{*}\RN\,.
	\]
	We denote the Levi-Civita connection on $M$ with respect to the induced
	metric $\gMetric=F^{*}\gMN$ by $\nabla$ and the corresponding curvature
	tensor by $\Rind$. \\
	
	The second fundamental tensor $\A$ of the graph $\Gamma(f)$ is
	defined as
	\[
		\A(v,w) \coloneqq \big( \widetilde{\nabla}\dd F\big)(v,w) = \nabla^{\gMN}_{\dd F(v)} \dd F(w) - \dd F(\nabla_{v}w) \,,
	\]
	where $v,w\in \Gamma(TM)$ and $\widetilde{\nabla}$ is the connection on $F^{*}T(M\times N)\otimes T^{*}M$ induced by the Levi-Civita connection.
	The trace of $\A$ with respect to the metric $\gMetric$ is called the
	\emph{mean curvature vector field} of $\Gamma(f)$ and it will be denoted by
	\[
		\Hv \coloneqq \tr \A \,.
	\]
	Note that $\Hv$ is a section in the normal bundle of the graph.
	If $\Hv$ vanishes identically, the graph is said to be minimal.
	A map $f:M\to N$ between Riemannian manifolds is called \emph{minimal},
	if its graph $\Gamma(f)$ is a minimal submanifold of the product
	space $(M\times N,\gMN)$ \cite{Sch93}.

	%%%%%%%%%%%%%%%%%%%%%%%%%%%%%%%%%%%%%%%%%%%%%%%%%%%%%%%%%%%%%%%%%%%%%%%%%%%%%%%%%%
	% SINGULAR VALUE DECOMPOSITION
	%%%%%%%%%%%%%%%%%%%%%%%%%%%%%%%%%%%%%%%%%%%%%%%%%%%%%%%%%%%%%%%%%%%%%%%%%%%%%%%%%%
	\subsection{Singular Value Decomposition}
	\label{sec:SVD}
	We recall the singular value decomposition theorem (see \cite[Section 3.2]{SHS13}
	and the references cited there). \\
	
	Fix a point $x\in M$, and let
	\[
		\svdf_{1}^{2}(x) \le \svdf_{2}^{2} \le \hdots \le \svdf_{m}^{2}(x)
	\]
	be the eigenvalues of $f^{*}\gN$ with respect to $\gM$. The
	corresponding values $\svdf_{i}\ge 0$, $i\in\{1,\hdots,m\}$, are called the
	\emph{singular values} of the differential $\dd f$ of $f$ and give rise
	to continuous functions on $M$. Let
	\[
		r \coloneqq r(x) \coloneqq \rank \dd f(x) \,.
	\]
	Obviously, $r\le \min\{m,n\}$ and $\svdf_{1}(x)=\cdots=\svdf_{m-r}(x)=0$. At the
	point $x$ consider an orthonormal basis $\{\alpha_{1},\hdots,\alpha_{m-r};\alpha_{m-r+1},\hdots,\alpha_{m}\}$
	with respect to $\gM$ which diagonalizes $f^{*}\gN$. Moreover, at $f(x)$ 
	consider a basis $\{\beta_{1},\hdots,\beta_{n-r};\beta_{n-r+1},\hdots,\beta_{n}\}$ that is
	orthonormal
	with respect to $\gN$, and such that
	\[
		\dd f(\alpha_{i}) = \svdf_{i}(x) \beta_{n-m+i}\,,
	\]
	for any $i\in\{m-r+1,\hdots,m\}$. The above procedure is called the
	\emph{singular value decomposition} of the differential $\dd f$. \\
	
	We construct a special basis for the tangent and the normal
	space of the graph in terms of the singular values. The vectors
	\[
		\widetilde{e}_i \coloneqq	\begin{cases} 
						\alpha_{i}\,, & 1\le i\le m-r\,, \\ 
						\frac{1}{\sqrt{1+\svdf_{i}^{2}(x)}}\big( \alpha_{i} \oplus \svdf_{i}(x)\beta_{n-m+i} \big)\,, &m-r+1 \le i \le m\,, 
					\end{cases}
	\]
	form an orthonormal basis with respect to the metric $\gMN$ of 
	the tangent space $\dd F(T_{x}M)$ of the graph $\Gamma(f)$ at $x$.
	It follows that with respect to the induced metric $\gind$, the vectors
	\[
		e_i \coloneqq \frac{1}{\sqrt{1+\svdf_i^2}} \alpha_i
	\]
	form an orthonormal basis of $T_xM$.
	Moreover,
	the vectors
	\[
		\xi_{i} \coloneqq	\begin{cases}
						\beta_{i}\,, & 1 \le i \le n-r\,, \\
						\frac{1}{\sqrt{1+\svdf_{i+m-n}^{2}(x)}}\bigl(-\svdf_{i+m-n}(x)\alpha_{i+m-n} \oplus \beta_{i}\bigr)\,, & n-r+1\le i\le n\,,
		          	\end{cases}
	\]
	form an orthonormal basis with respect to $\gMN$ of the
	normal space $\mathcal{N}_{x}M$ of the graph $\Gamma(f)$ at the point
	$x$. From the formulae above, we deduce that
	\[
		\sTensor_{M\times N}(\widetilde{e}_i,\widetilde{e}_j) = \sind(e_i,e_j) = \frac{1-\svdf_{i}^{2}(x)}{1+\svdf_{i}^{2}(x)}\delta_{ij} \,, \qquad 1\le i,j\le m \,.
	\]
	Therefore, the eigenvalues of the $2$-tensor $\sTensor$ with respect to $\gMetric$
	are given by
	\[
		\frac{1-\svdf_{1}^{2}(x)}{1+\svdf_{1}^{2}(x)} \ge \cdots \ge \frac{1-\svdf_{m-1}^{2}(x)}{1+\svdf_{m-1}^{2}(x)} \ge \frac{1-\svdf_{m}^{2}(x)}{1+\svdf_{m}^{2}(x)} \,.
	\]
	Moreover,
	\[
		\sTensor_{M\times N}(\xi_{i},\xi_{j}) =
			\begin{cases}
				- \delta_{ij}\,, & 1 \le i \le n-r \,, \\
				- \frac{1-\svdf_{i+m-n}^{2}(x)}{1+\svdf_{i+m-n}^{2}(x)}\delta_{ij} \,, & n-r+1\le i \le n\,,
			\end{cases}
	\]
	and
	\[
		\sTensor_{M\times N}(\widetilde{e}_{m-r+i},\xi_{n-r+j}) = - \frac{2\svdf_{m-r+i}(x)}{1+\svdf_{m-r+i}^{2}(x)} \delta_{ij} \,,\qquad 1\le i,j\le r \,.
	\]

	%%%%%%%%%%%%%%%%%%%%%%%%%%%%%%%%%%%%%%%%%%%%%%%%%%%%%%%%%%%%%%%%%%%%%%%%%%%%%%%%%%
	% Bernstein-type Theorems
	%%%%%%%%%%%%%%%%%%%%%%%%%%%%%%%%%%%%%%%%%%%%%%%%%%%%%%%%%%%%%%%%%%%%%%%%%%%%%%%%%%
	\section{Bernstein-type Theorems}
	\label{sec:Proofs}

	In order to show theorem \ref{thm:ThmA}, we need to control the eigenvalues of the tensor $\sTensor$.
	For this, following the ideas of \cite{SHS13}, let us define the tensor
	\[
		\Phi_{c} \coloneqq \sTensor - \frac{1-c}{1+c} \gMetric \,.
	\]
	The differential equation satisfied by $\Phi_c$
	was derived in \cite[Lemma 3.2]{SHS13} and is given by
	\begin{align*}
		\big( \Delta \Phi_{c} \big) (v,w) &= \sTensor_{M\times N}\big( \nabla_{v} \Hv, \dd F(w) \big) + \sTensor_{M\times N}\big( \nabla_{w}\Hv, \dd F(v) \big) \\
			& \qquad + 2 \frac{1-c}{1+c} \gMN\big( \Hv, \A(v,w) \big) \\
			& \qquad + \Phi_{c}\big( \Ric v, w \big) + \Phi_{c} \big( \Ric w, v \big) \\
			& \qquad + 2 \sum_{k=1}^{m} \left( \sTensor_{M\times N} - \frac{1-c}{1+c} \gMN \right) \big( \A(e_{k},v), \A(e_{k},w) \big) \\
			& \qquad + \frac{4}{1+c} \sum_{k=1}^{m} \big( f^{*}\RN(e_{k},v,e_{k},w) - c \RM(e_{k},v,e_{k},w) \big) \,,
	\end{align*}
	where
	\[
		\Ric v \coloneqq - \sum_{k=1}^{m} \Rind(e_{k},v) e_{k}
	\]
	is the Ricci operator on $(M,\gMetric)$ and $\{e_{1},\hdots,e_{m}\}$ is any orthonormal
	frame with respect to the induced metric $\gMetric$. \\
	
	Now let $f:M\to N$ be minimal, i.\,e.\ $\Hv=0$. 
	Setting 
	\begin{align*}
		\Psi_c(\vartheta)(v,w) &\coloneqq - \vartheta(\Ric v,w) - \vartheta(\Ric w,v) \\
			& \qquad - 2 \sum_{k=1}^{m} \left( \sTensor_{M\times N} - \frac{1-c}{1+c} \gMN \right) \big( \A(e_{k},v), \A(e_{k},w) \big) \\
			& \qquad - \frac{4}{1+c} \sum_{k=1}^{m} \big( f^{*}\RN(e_{k},v,e_{k},w) - c \RM(e_{k},v,e_{k},w) \big) \,,
	\end{align*}
	we can write
	\[
		\Delta \Phi_c + \Psi_c(\Phi_c) = 0\,.
	\]

	To estimate the terms in $\Psi_c$ which contain the second 
	fundamental tensor, we have the following statement.

	\begin{lemma}
		\label{lem:AEst}
		Let $\eta\in T_{p}^{\perp}M$ be a normal vector on the graph.
		Then for any $c\ge\svdf_{m}^{2}$ the estimate
		\[
			\sMN(\eta,\eta) \le \frac{c-1}{1+c} \gMN(\eta,\eta)
		\]
		holds. In particular,
		\[
			\sMN\big( \A(v,w), \A(v,w) \bigr) \le \frac{c-1}{1+c} \|\A(v,w)\|^2 \,.
		\]
	\end{lemma}
	
	\begin{proof}
		Using the formulae for the singular values above, we calculate
		\[
			\sTensor_{M\times N}(\xi_{i},\xi_{j}) = - \delta_{ij} \le - \frac{1-c}{1+c} \delta_{ij} \,, \qquad 1\le i,j\le n-r
		\]
		and
		\begin{align*}
			\sTensor_{M\times N}(\xi_{i},\xi_{j}) &= - \frac{1-\svdf_{i+m-n}^{2}}{1+\svdf_{i+m-n}^{2}} \delta_{ij} \\
				&\le - \frac{1-c}{1+c} \delta_{ij} \,, \qquad \qquad n-r+1 \le i,j\le n \,. \qedhere
		\end{align*}
	\end{proof}

	The terms involving the curvatures can be decomposed in the following way.	
		
	\begin{lemma}
		\label{lem:PhiNullEV}
		Let $\{e_1,\dotsc,e_m\}$ be a local $\gind$-orthonormal frame. Then for any $e_l$ we have
		\begin{align*}
			& 2 \sum_{k=1}^{m}\big( f^{*}\RN(e_{k},e_{l},e_{k},e_{l}) - c \RM(e_{k},e_{l},e_{k},e_{l}) \big) \\
			& \quad = -2 \sum_{k\ne l} f^{*}\gN(e_{k},e_{k}) \biggl\{ \big( \sigma - \sec_{N}(\dd f(e_{k})\wedge\dd f(e_{l})) \big) f^*\gN(e_l,e_l) \\
					& \qquad \qquad \qquad \qquad \qquad \qquad + \sigma \Phi_c(e_l,e_l) \biggr\} \\
				& \qquad - c \gM(e_l,e_l) \sum_{k\ne l} \Phi_c(e_k,e_k) \bigl( \sec_M(e_k\wedge e_l) - \sigma \bigr) \\
				& \qquad - \frac{2c}{1+c} \bigl( \Ric_M(e_l,e_l) - (m-1) \sigma \gM(e_l,e_l) \bigr)\\
				& \qquad - \frac{2\sigma c}{1+c} \bigl( \tr(\sind) - \sind(e_l,e_l) \bigr) \\
				& \qquad - \frac{\sigma(1+c)}{2} \Phi_c(e_l,e_l) \left( \tr\bigl(\Phi_c\bigr) - \Phi_c(e_l,e_l) \right) \,.
		\end{align*}
	\end{lemma}
	
	\begin{proof}
		We calculate
		\begin{align*}
			& 2 \sum_{k=1}^{m}\big( f^{*}\RN(e_{k},e_{l},e_{k},e_{l}) - c \RM(e_{k},e_{l},e_{k},e_{l}) \big) \\
			& \quad  = 2 \sum_{k\ne l} \sec_{N}\big( \dd f(e_{k})\wedge\dd f(e_{l}) \big) f^{*}\gN(e_{k},e_{k}) f^{*}\gN(e_{l},e_{l}) \\
				& \qquad - 2c \sum_{k\ne l} \sec_{M}(e_{k}\wedge e_{l}) \gM(e_{k},e_{k}) \gM(e_{l},e_{l}) \\
			& \quad = -2 \sum_{k\ne l} f^{*}\gN(e_{k},e_{k})\biggl\{ \bigl( \sigma - \sec_{N}(\dd f(e_{k})\wedge\dd f(e_{l})) \bigr) f^{*}\gN(e_{l},e_{l}) \\
				& \qquad \qquad \qquad \qquad \qquad \qquad + \sigma\biggl( \gM(e_{l},e_{l}) - f^{*}\gN(e_{l},e_{l}) - \frac{1-c}{1+c} \biggr) \biggr\} \\
				& \qquad + 2\sigma \left( \gM(e_{l},e_{l}) - \frac{1-c}{1+c} \right) \sum_{k\ne l} f^{*}\gN(e_{k},e_{k}) \\
				& \qquad - 2c\sum_{k\ne l} \bigl( \sec_M(e_k\wedge e_l) - \sigma \bigr) \gM(e_k,e_k) \gM(e_l,e_l) \\
				& \qquad - 2c\sigma \sum_{k\ne l} \gM(e_k,e_k) \gM(e_l,e_l) \,.
		\end{align*}
		Since
		\[
			2 \gM(e_k,e_k) = \sind(e_k,e_k) + 1 = \left( \sind(e_k,e_k) - \frac{1-c}{1+c} \right) + \frac{2}{1+c} \,,
		\]
		the term involving the sectional curvatures of $M$ evaluates to
		\begin{align*}
				& - 2c\sum_{k\ne l} \bigl( \sec_M(e_k\wedge e_l) - \sigma \bigr) \gM(e_k,e_k) \gM(e_l,e_l) \\
				& \quad = - c \gM(e_l,e_l) \sum_{k\ne l} \left( \sind(e_k,e_k) - \frac{1-c}{1+c} \right) \bigl( \sec_M(e_k\wedge e_l) - \sigma \bigr) \\
					& \qquad - \frac{2c}{1+c} \bigl( \Ric_M(e_l,e_l) - (m-1) \sigma \gM(e_l,e_l) \bigr) \,.
		\end{align*}
		Further, using
		\[
			2 \gM(e_k,e_k) = 1 + \sind(e_k,e_k) \,, \qquad 2 f^*\gN(e_k,e_k) = 1 - \sind(e_k,e_k) \,,
		\]
		for the remaining terms we get
		\begin{align*}
				& 2\sigma \left( \gM(e_{l},e_{l}) - \frac{1-c}{1+c} \right) \sum_{k\ne l} f^{*}\gN(e_{k},e_{k}) - 2c\sigma \sum_{k\ne l} \gM(e_k,e_k) \gM(e_l,e_l) \\
					& \quad = \frac{2\sigma c}{1+c} \sum_{k\ne l} \bigl( f^*\gN(e_k,e_k) - \gM(e_k,e_k) \bigr) \\
						& \qquad \quad + 2\sigma \left( \gM(e_l,e_l) - \frac{1}{1+c} \right) \sum_{k\ne l} \bigl( f^*\gN(e_k,e_k) - c\gM(e_k,e_k) \bigr) \\
					& \quad = - \frac{2\sigma c}{1+c} \bigl( \tr(\sind) - \sind(e_l,e_l) \bigr) \\
						& \qquad \quad + \sigma \left( \sind(e_l,e_l) - \frac{1-c}{1+c} \right) \sum_{k\ne l}\left( \frac{1-c}{2} - \frac{1+c}{2} \sind(e_k,e_k) \right) \\
					& \quad = - \frac{2\sigma c}{1+c} \bigl( \tr(\sind) - \sind(e_l,e_l) \bigr) \\
						& \qquad \quad + \frac{1+c}{2} \sigma \Phi_c(e_l,e_l) \left( (m-1)\frac{1-c}{1+c} - \tr(\sind) + \sind(e_l,e_l) \right) \,.
		\end{align*}
		Collecting all terms and using $\sind=\Phi_c+\frac{1-c}{1+c}\gind$, the claim of the lemma follows.
	\end{proof}
	
	\begin{corollary}
		\label{cor:Psi}
		Let $\{e_1,\dotsc,e_m\}$ be a local $\gind$-orthonormal frame. Then for any $e_l$ we have
		\begin{align*}
			& 2 \sum_{k=1}^{m}\big( f^{*}\RN(e_{k},e_{l},e_{k},e_{l}) - c \RM(e_{k},e_{l},e_{k},e_{l}) \big) \\
			& \quad = -2 \sum_{k\ne l} f^{*}\gN(e_{k},e_{k}) \big( \sigma - \sec_{N}(\dd f(e_{k})\wedge\dd f(e_{l})) \big) f^*\gN(e_l,e_l) \\
				& \qquad - c \gM(e_l,e_l) \sum_{k\ne l} \Phi_c(e_k,e_k) \bigl( \sec_M(e_k\wedge e_l) - \sigma \bigr) \\
				& \qquad - \frac{2c}{1+c} \bigl( \Ric_M(e_l,e_l) - (m-1) \sigma \gM(e_l,e_l) \bigr)\\
				& \qquad - \frac{2c\sigma }{1+c} \tr(\sind) \\
				& \qquad + \frac{\sigma(1-c)}{2} \Phi_c(e_l,e_l) \left( \tr\bigl(\Phi_c\bigr) - \Phi_c(e_l,e_l) \right) \\
				& \qquad - \frac{2c\sigma}{1+c} \left( (m-2) \Phi_c(e_l,e_l) - \frac{1-c}{1+c}  \right)\,.
		\end{align*}
	\end{corollary}
	
	\begin{proof}
		We note that
		\begin{align*}
			& -2 \sigma \Phi_c(e_l,e_l) \sum_{k\ne l} f^*\gN(e_k,e_k) \\
			& \quad = - \sigma \Phi_c(e_l,e_l) \sum_{k\ne l} \bigl( 1 - \sind(e_k,e_k) \bigr) \\
				&\quad = - \sigma \Phi_c(e_l,e_l) \bigl( (m-1) - \tr(\sind) + \sind(e_l,e_l) \bigr) \\
				&\quad = - \sigma \Phi_c(e_l,e_l) \biggl( (m-1) - \tr(\Phi_c) - m\frac{1-c}{1+c} + \Phi_c(e_l,e_l) + \frac{1-c}{1+c} \biggr) \\
				&\quad = \sigma \Phi_c(e_l,e_l) \Bigl( \tr(\Phi_c) - \Phi_c(e_l,e_l) \Bigr) - (m-1) \frac{2c\sigma}{1+c} \Phi_c(e_l,e_l) \,.
		\end{align*}
		Then from lemma \ref{lem:PhiNullEV} and using $\sind=\Phi_c+\frac{1-c}{1+c}\gind$ we get
		\begin{align*}
			& 2 \sum_{k=1}^{m}\big( f^{*}\RN(e_{k},e_{l},e_{k},e_{l}) - c \RM(e_{k},e_{l},e_{k},e_{l}) \big) \\
				& \quad = -2 \sum_{k\ne l} f^*\gN(e_k,e_k) \bigl( \sigma - \sec_N(\dd f(e_k) \wedge \dd f(e_l) \bigr) f^*\gN(e_l,e_l) \\
				& \qquad + \sigma \Phi_c(e_l,e_l) \Bigl( \tr(\Phi_c) - \Phi_c(e_l,e_l) \Bigr) - (m-1) \frac{2c \sigma}{1+c} \Phi_c(e_l,e_l) \\
				& \qquad - c\gM(e_l,e_l) \sum_{k\ne l} \Phi_c(e_k,e_k) \bigl( \sec_M(e_k\wedge e_l) - \sigma \bigr) \\
				& \qquad - \frac{2c}{1+c} \bigl( \Ric_M(e_l,e_l) - (m-1) \sigma \gM(e_l,e_l) \bigr) \\
				& \qquad - \frac{2\sigma c}{1+c} \tr(\sind) + \frac{2\sigma c}{1+c} \left( \Phi_c(e_l,e_l) + \frac{1-c}{1+c} \right) \\
				& \qquad - \frac{\sigma(1+c)}{2} \Phi_c(e_l,e_l) \bigl( \tr(\Phi_c) - \Phi_c(e_l,e_l) \bigr) \,.
		\end{align*}
		Rearranging the terms, the claim follows.
	\end{proof}

	%%%%%%%%%%%%%%%%%%%%%%%%%%%%%%%%%%%%%%%%%%%%%%%%%%%%%%%%%%%%%%%%%%%%%%%%%%%%%%%%%%
	% PROOF OF THEOREM A
	%%%%%%%%%%%%%%%%%%%%%%%%%%%%%%%%%%%%%%%%%%%%%%%%%%%%%%%%%%%%%%%%%%%%%%%%%%%%%%%%%%
	\subsection{Proof of Theorem \ref{thm:ThmA}}
	From the compactness of $M$ we get that the largest singular 
	value $\lambda_m$ of $\dd f$ attains
	its maximum at some point on $M$. 
	Let $\svdf_{0}$ denote this maximum,
	\[
		\svdf_{0}^{2} \coloneqq \max_{x\in M} \svdf_{m}^{2}(x) \,.
	\]
	Then $\Phi_{\svdf_{0}^{2}}$ is non-negative definite, and at the point where $\svdf_{m}^{2}$
	attains its maximum, we know that $e_{m}$ is a null-eigenvector of $\Phi_{\svdf_{0}^{2}}$.
		
	\begin{lemma}
		\label{lem:StrictLengthDecr}
		Assume $\sec_N\le\sigma\le\sec_M$ for some $\sigma>0$ and assume
		$\svdf_{0}^2 < 1$. Then the map $f:M\to N$ is constant.
	\end{lemma}
	
	\begin{proof}
		We first show that 
		$\Psi_{\svdf_0^2}$ satisfies the null-eigenvector condition. Let
		$v$ be a null-eigenvector of the positive semi-definite tensor $\vartheta$
		and extend it to a local $\gind$-orthonormal frame $\{e_1,\dotsc,e_l=v,\dotsc,e_m\}$.
		Since $\svdf_0^2<1$, the map $f$ is strictly length-decreasing and accordingly it is
		$\sind(w,w)>0$ for any vector $w\ne 0$. 
		Also noting that $\Phi_{\svdf_0^2}(w,w)\ge 0$ for any $w$
		by construction and
		\begin{align*}
			\tr(\sind) - \sind(v,v) &= \sum_{k\ne l} \underbrace{\sind(e_k,e_k)}_{>0} > 0\,, \\
			\tr(\Phi_{\svdf_0^2}) - \Phi_{\svdf_0^2}(v,v) &= \sum_{k\ne l} \underbrace{\Phi_{\svdf_0^2}(e_k,e_k)}_{\ge 0}  \ge 0 \,,
		\end{align*}
		we calculate
		\begin{align*}
			\Psi_{\svdf_0^2} (\vartheta)(v,v) &\stackrel{\text{Lem.\ \ref{lem:AEst}}}{\ge}  \underbrace{4 \frac{1-\svdf_0^2}{1+\svdf_0^2} \sum_{k=1}^m \|\A(e_k,v)\|^2}_{\ge 0} \\
					& \qquad \qquad - \frac{4}{1+\svdf_0^2} \sum_{k=1}^m \bigl( f^*\RN(e_k,v,e_k,v) - \svdf_0^2 \RM(e_k,v,e_k,v) \bigr) \\
				&\stackrel{\hphantom{\text{Lem.\ \ref{lem:AEst}}}}{\ge} \frac{4\sigma \svdf_0^2}{(1+\svdf_0^2)^2}\left( \tr(\sind) - \sind(v,v) \right) \\
					& \qquad \qquad + \sigma \Phi_{\svdf_0^2}(v,v) \bigl( \tr(\Phi_{\svdf_0^2}) - \Phi_{\svdf_0^2}(v,v) \bigr) \\
				&\stackrel{\hphantom{\text{Lem.\ \ref{lem:AEst}}}}{\ge} 0\,,
		\end{align*}
		where we have used lemma \ref{lem:PhiNullEV} and the curvature assumptions.
		Thus, the null-eigenvector condition is satisfied. \\
		
		By the definition
		of $\svdf_0^2$, at the point $x\in M$ where the largest singular value $\svdf_m$
		of $\dd f$ attains its maximal value, we have $\Phi_{\svdf_0^2}(e_m,e_m)=0$,
		so that $\Phi_{\svdf_0^2}$ has a null-eigenvector at that point.
		Consequently, by the strong
		maximum principle (see theorem \ref{thm:SMP}),
		$\Phi_{\svdf_0^2}$  
		has a null-eigenvector at any point of $M$. \\
		
		We will now show that if $\Phi_{\svdf_0^2}$ admits a null-eigenvector at some point $x$,
		then $\Phi_{\svdf_0^2}$ vanishes at $x$. Since by definition $\Phi_{\svdf_0^2}\ge 0$,
		we can apply the second derivative test criterion at an arbitrary point $x\in M$.
		Using lemmas \ref{lem:AEst} and \ref{lem:PhiNullEV}, at $x$ we calculate
		\begin{align*}
			0 &\stackrel{\substack{\hphantom{\text{Lem.\ \ref{lem:PhiNullEV}}}\\\text{min}}}{\le} (\Delta \Phi_{\svdf_0^2})(e_m,e_m) \\
				&\stackrel{\text{Lem.\ \ref{lem:AEst}}}{\le} \underbrace{ 4 \frac{\svdf_0^2-1}{1+\svdf_0^2} \sum_{k=1}^m \| \A(e_k,e_m) \|^2}_{\le 0} \\
					& \qquad \qquad + \frac{4}{1+\svdf_0^2} \sum_{k=1}^m \bigl( f^*\RN(e_k,e_m,e_k,e_m) - \svdf_0^2 \RM(e_k,e_m,e_k,e_m) \bigr) \\
				& \stackrel{\text{Lem.\ \ref{lem:PhiNullEV}}}{\le} \underbrace{- \frac{4\sigma\svdf_0^2}{(1+\svdf_0^2)^2}}_{\le 0} \sum_{k\ne m} \underbrace{\sind(e_k,e_k)}_{>0} = 0 \,.
		\end{align*}
		Consequently, since $\sigma>0$, 
		we must have $\svdf_0^2=0$. Since $\svdf_0^2$ is the maximum
		of the largest singular value, it is $0 \le \svdf_1^2(x) = \dotsb = \svdf_m^2(x) \le \svdf_0^2 = 0$.
	\end{proof}
	
	To prove the remaining case of theorem \ref{thm:ThmA},
	we only need to consider the case $\svdf_0^2\ge 1$.

	\begin{lemma}
		\label{lem:PhiNullEV2}
		Assume $\sec_{N} \le \sigma \le \sec_{M}$ for some $\sigma > 0$.
		Further, assume $\tr(\sind)\ge 0$ and $\lambda_0^2\ge 1$.
		If there exists $\kappa>1$, such that
		$f^{*}\gN < \kappa^{2} \gM$ and
		\[
			\|\A\|^{2} \le \frac{\kappa^{2} \sigma}{\kappa^{4}-1} \tr(\sTensor) \,,
		\]
		then $\Psi_{\svdf_{0}^{2}}$ satisfies the null-eigenvector condition.
	\end{lemma}
	
	\begin{proof}
		Let
		$v$ be a null-eigenvector of the positive semi-definite tensor $\vartheta$
		and extend it to a local $\gind$-orthonormal frame $\{e_1,\dotsc,e_l=v,\dotsc,e_m\}$.
		Note that $\Phi_{\svdf_0^2}(v,v)\ge 0$ by construction
		and consequently
		\[
			\tr(\Phi_{\svdf_0^2}) - \Phi_{\svdf_0^2}(v,v) = \sum_{k\ne l} \underbrace{\Phi_{\svdf_0^2}(e_k,e_k)}_{\ge 0} \ge 0 \,.
		\]
		From the curvature assumptions, lemma \ref{lem:AEst} and corollary \ref{cor:Psi}
		we obtain 
		\begin{align*}
			\Psi_{\svdf_0^2}(\vartheta)(v,v) &\stackrel{\text{Lem.\ \ref{lem:AEst}}}{\ge} -4 \frac{\svdf_0^2-1}{1+\svdf_0^2} \sum_{k=1}^m \| \A(e_k,v) \|^2 \\
					&\qquad \qquad - \frac{4}{1+\svdf_0^{2}} \sum_{k=1}^{m}\big( f^{*}\RN(e_{k},v,e_{k},v) - \svdf_0^{2}\RM(e_{k},v,e_{k},v) \big) \\
				& \stackrel{\substack{\hphantom{Lem.\ \ref{lem:AEst}}\\\svdf_0^2\ge 1\\\text{Cor.\ \ref{cor:Psi}}}}{\ge} - \frac{4}{1+\svdf_0^{2}}\left\{ (\svdf_0^{2}-1) \|\A\|^{2} - \frac{\svdf_0^{2}}{1+\svdf_0^{2}} \sigma \tr(\sind) \right\} \,.
		\end{align*}
		If $\lambda_0^2=1$, then
		\[
			\Psi_{\svdf_0^2}(\vartheta)(v,v) \ge \sigma \tr(\sind) \ge 0 \,.
		\]
		so that the null-eigenvector condition is satisfied. If $\lambda_0^2>1$,
		also using $\svdf_{0}^{2} < \kappa^{2}$, it is
		\[
			\frac{\kappa^{2}}{\kappa^{4}-1} < \frac{\svdf_{0}^{2}}{\svdf_{0}^{4}-1}
		\]
		and accordingly
		\begin{align*}
			\Psi_{\svdf_0^2}(\vartheta)(v,v) & \ge - \frac{4}{1+\svdf_0^2}\left\{ (\svdf_0^{2}-1) \|\A\|^{2} - \frac{\svdf_0^{2}}{1+\svdf_0^{2}} \sigma \tr(\sTensor) \right\} \\
				& = - 4 \frac{\svdf_0^2-1}{1+\svdf_0^{2}}\left\{ \|\A\|^{2} - \frac{\svdf_0^{2}}{\svdf_0^4-1} \sigma \tr(\sTensor) \right\} \\
				& \ge - 4 \frac{\svdf_0^2-1}{1+\svdf_0^2} \left\{ \frac{\kappa^2}{\kappa^4-1} - \frac{\svdf_0^2}{\svdf_0^4-1} \right\} \sigma \tr(\sind) \\
				&\ge 0 \,,
		\end{align*}
		so that the null-eigenvector condition is satisfied.
	\end{proof}
	
	Let us now finish the proof of theorem \ref{thm:ThmA}.	
	By construction, the non-negative definite tensor
	$\Phi_{\svdf_0^2}$ admits a null-eigenvector at some point.
	By lemma \ref{lem:PhiNullEV2}, we can apply the strong
	maximum principle (see theorem \ref{thm:SMP})
	to conclude that
	$\Phi_{\svdf_{0}^{2}}$ admits a null-eigenvector everywhere on $M$,
	i.\,e.\ $\Phi_{\svdf_0^2}$ attains its minimal value at every point on $M$.
	Thus, we may apply the second derivative test criterion to $\Phi_{\svdf_0^2}$ at
	an arbitrary point $x\in M$. At $x$ consider a basis $\{e_1,\dotsc,e_m\}$,
	orthonormal with respect to $\gind$ consisting of eigenvectors of $\Phi_{\svdf_0^2}$,
	such that $e_m$
	is a null-eigenvector of $\Phi_{\svdf_0^2}$ and $\svdf_m^2(x)=\svdf_0^2$.
	From lemma \ref{lem:AEst} and corollary \ref{cor:Psi}
	we conclude
	\begin{align}
		\nonumber 0 &\stackrel{\text{min}}{\le} \big( \Delta \Phi_{\svdf_{0}^{2}} \big)(e_{m},e_{m}) \\
			\nonumber & \stackrel{\text{\phantom{min}}}{\le} \underbrace{\frac{4}{1+\svdf_{0}^{2}}}_{> 0} \Biggl( \underbrace{ (\svdf_{0}^{2}-1) \|\A\|^{2} - \frac{\svdf_{0}^{2}}{1+\svdf_{0}^{2}}\sigma \tr(\sTensor)}_{\le 0} \Biggr) \\
				\nonumber &\quad - \underbrace{\frac{4}{1+\svdf_0^2}}_{> 0} \sum_{k\ne m} \underbrace{f^{*}\gN(e_{k},e_{k})}_{\ge 0} \big( \underbrace{\sigma-\sec_{N}(\dd f(e_{k})\wedge \dd f(e_{m}))}_{\ge 0}\big) \\
				\nonumber &\quad - \underbrace{\frac{4\svdf_{0}^{2}}{(1+\svdf_0^2)^2}}_{\ge 0} \sum_{k\ne m} \underbrace{\Phi_{\svdf_0^2}(e_m,e_m)}_{\ge 0} \big( \underbrace{\sec_{M}(e_{k}\wedge e_{m}) - \sigma}_{\ge 0}\big) \\
				\nonumber &\quad - \underbrace{\frac{4\svdf_{0}^{2}}{(1+\svdf_{0}^{2})^{2}}}_{\ge 0}\Big( \underbrace{\Ric_{M}(e_{m},e_{m}) - (m-1)\sigma \gM(e_{m},e_{m})}_{\ge 0}\Big) \\
				\nonumber &\quad + \underbrace{\frac{1-\svdf_0^2}{1+\svdf_0^2}}_{\le 0} \sigma \underbrace{\Phi_{\svdf_0^2}(e_m,e_m)}_{\ge 0} \Bigl( \underbrace{\tr(\Phi_{\svdf_0^2}) - \Phi_{\svdf_0^2}(e_m,e_m)}_{\ge 0} \Bigr) \\
				\label{eq:LapPhi} &\quad - \underbrace{\frac{4\svdf_0^2}{(1+\svdf_0^2)^2}}_{> 0} \sigma \biggl( \underbrace{(m-2) \Phi_{\svdf_0^2}(e_m,e_m)}_{\ge 0} \underbrace{-\frac{1-\svdf_0^2}{1+\svdf_0^2}}_{\ge 0} \biggr) = 0 \,.
	\end{align}
	Since all terms have the same sign, they have to vanish independently.
	Using $\sigma>0$ and $\lambda_0^2\ge 1$, the second line and the
	assumptions on $\|\A\|^2$ imply
	\begin{align*}
		0 &= \bigl( \svdf_0^2-1 \bigr) \|\A\|^2 - \frac{\svdf_0^2}{1+\svdf_0^2} \sigma \tr(\sind) \\
			& \le \left( \bigl( \svdf_0^2-1 \bigr) \frac{\kappa^2}{\kappa^4-1} - \frac{\svdf_0^2}{1+\svdf_0^2} \right) \sigma \tr(\sind) \stackrel{\svdf_0^2<\kappa^2}{\le} 0 \,.
	\end{align*}
	Since $\sigma>0$, this yields $\tr(\sind)=0$. From
	\[
		0 = - \underbrace{\frac{4\svdf_0^2}{(1+\svdf_0^2)^2}}_{> 0} \sigma \biggl( \underbrace{(m-2) \Phi_{\svdf_0^2}(e_m,e_m)}_{\ge 0} \underbrace{-\frac{1-\svdf_0^2}{1+\svdf_0^2}}_{\ge 0} \biggr)
	\]
	we infer $\svdf_0^2=1$.
	Consequently, $\tr(\sind)=0$
	and $\svdf_1^2(x) \le \dotsb \le \svdf_m^2(x)=\lambda_0^2 = 1$ force
	all singular values to be equal to one at any point $x\in M$, 
	$\svdf_1^2(x) = \dotsb = \svdf_m^2(x) = 1$. Now the assumption on the second
	fundamental form reads
	\[
		\|\A\|^2 \le \frac{\kappa^2 \sigma}{\kappa^4-1} \tr(\sind) = 0 \,,
	\]
	so that $\Gamma(f)$ is totally geodesic. Further, it follows that $\gind=2\gM$
	and the second fundamental form also satisfies the equation
	\[
		0 = \A = \A_I \oplus \A_f \,,
	\]
	where $\A_I$ is the second fundamental form of the map
	$I:(M,\gind)\to (M,\gM)$, and $\A_f$ is the second fundamental form
	of the map $f:(M,\gind)\to (N,\gN)$. Consequently, $\A_f=0$, so that
	$f$ is a totally geodesic isometric immersion. \\

	Since all singular values are the same and all vectors are null-eigenvectors
	of $\Phi_{\svdf_0^2}$, we can evaluate $0 = (\Delta \Phi_{\svdf_0^2})(v,v)$
	for an arbitrary vector $v$. Then 
	the third line in \eqref{eq:LapPhi} yields $\sec_N=\sigma$ on $\dd f(TM)$, and
	the fifth line implies $\Ric_M=\sigma\gM$. In view of the curvature
	assumptions, the latter means $\sec_M=\sigma$.
	Thus, the claim of the theorem follows. \hfill \qed
		
	\subsection{Proof of Corollary \ref{cor:CorB}}
	
	Using the assumptions on the dimensions and the notation
	from section \ref{sec:SVD}, we have $r=r(x)=\rank(\dd f) \le \min\{m,n\}=n$
	and $\svdf_j(x) = 0$ for $j=1,\dotsc,m-r$. Consequently,
	\begin{align*}
		\tr(\sind) &= \underbrace{\sum_{j=1}^{m-r}1}_{\ge m-n} + \underbrace{\sum_{j=m-r+1}^m \frac{1-\svdf_{n-m+j}^2(x)}{1+\svdf_{n-m+j}^2(x)}}_{>-r} > m-n-r \ge m-2n \ge 0 \,,
	\end{align*}
	so that condition \eqref{eq:trCon} of theorem \ref{thm:ThmA} is fulfilled.
	Condition \eqref{eq:sCondition4} is satisfied
	by assumption, so that we can draw the conclusions of theorem \ref{thm:ThmA}.
	Further, since the above inequality for $\tr(\sind)$ is strict,
	the map $f:M\to N$ must be constant and the claim follows. \hfill \qed

	\section{Discussion}
	\label{sec:Discussion}
	
	We conclude by giving some remarks on the assumptions made in the theorem.
	
	\begin{remark}
		\begin{enumerate}
			\item Since $M$ is compact, a constant $\kappa^{2}>1$ satisfying
				the inequality $f^{*}\gN < \kappa^{2} \gM$ always exists. 
			\item If $\dim M\ge 3$, the condition on the trace
				\eqref{eq:trCon} is strictly weaker than the area-decreasing condition. 
				In view of Eq.\ \eqref{eq:sCondition4}, this is compensated by an additional requirement on the second fundamental form.
		\end{enumerate}
	\end{remark}
	
	\begin{remark}
		In some situations, we can exclude case \ref{it:thmAii} of the theorem. For example,
		if $\dim M > \dim N$, a map $f$ satisfying the assumptions of theorem \ref{thm:ThmA}
		cannot be an isometric immersion, so that $f$ necessarily has to be constant.
	\end{remark}
	
	To draw the conclusion of theorem \ref{thm:ThmA}, the assumptions
	can be weakened in various situations. Below, we will always assume
	the curvature conditions of theorem \ref{thm:ThmA} to hold.
	
	\begin{remark}[Low dimensions]
		In low dimensions, one can remove the assumption on the second fundamental form.
		\begin{enumerate}
		\item $\dim N=1$. In this case, any smooth minimal map $f:M\to N$ is constant \cite[Theorem C]{SHS13}.
		\item $\dim M=2$. Here, the condition $\tr(\sind)\ge 0$ means that $f$ is a weakly area-decreasing map. 
			Then \cite[Theorem B]{SHS13} implies that $f$ is either constant, or it is an isometric immersion on a non-empty, closed subset $D$ of $M$ and
			strictly area-decreasing on the complement of $D$. 
		\item $\dim M=\dim N=2$. Let us  consider the Jacobian of the projection map $\pi_M : \Gamma(f)\to M$,
			which may be expressed as $v\coloneqq\star\Omega_{\gind}$, where $\Omega_{\gind}$ is the
			volume form on $\Gamma(f)$ induced by the metric $\gind$ and $\star:\Omega^k(M)\to\Omega^{2-k}(M)$
			is the Hodge star with respect to the induced metric $\gind$.
			The differential equation satisfied by $v$
			is essentially calculated in \cite{Wang02}*{Proposition 3.1} and is given by
			\begin{align*}
				\Delta \ln v &= -\|\A\|^2 - \svdf_1^2 \Bigl\{ (\A_{11}^1)^2 + (\A_{12}^1)^2 \Bigr\} - \svdf_2^2 \Bigl\{ (\A_{12}^2)^2 + (\A_{22}^2)^2 \Bigr\} \\
					& \quad - 2 \svdf_1\svdf_2 \Bigl\{ \A_{11}^2 \A_{21}^1 + \A_{12}^2 \A_{22}^1 \Bigr\} \\
					& \quad - \frac{1}{(1+\svdf_1^2)(1+\svdf_2^2)} \Bigl\{ (\svdf_1^2+\svdf_2^2) \sec_M - 2 \svdf_1^2\svdf_2^2 \sec_N \Bigr\} \,.
			\end{align*}
			From the minimality of $\Gamma(f)$ we also get 
			$\A_{11}^2=-\A_{22}^2$ and $\A_{22}^1=-\A_{11}^1$, so that we estimate
			\begin{align*}
				\Delta \ln v &\le -\|\A\|^2 - \svdf_1^2 \Bigl\{ (\A_{11}^1)^2 + (\A_{12}^1)^2 \Bigr\} - \svdf_2^2 \Bigl\{ (\A_{12}^2)^2 + (\A_{22}^2)^2 \Bigr\} \\
					& \quad + 2 |\svdf_1\svdf_2| \Bigl\{ \bigl|\A_{22}^2 \A_{21}^1\bigr| + \bigl|\A_{12}^2 \A_{11}^1 \bigr| \Bigr\} \\
					& \quad - \frac{1}{(1+\svdf_1^2)(1+\svdf_2^2)} \Bigl\{ (\svdf_1^2+\svdf_2^2) \sec_M - 2 \svdf_1^2\svdf_2^2 \sec_N \Bigr\} \\
					&= -\|\A\|^2 - \Bigl( \bigl|\svdf_1\A_{11}^1\bigr| - \bigl|\svdf_2\A_{12}^2 \bigr| \Bigr)^2 - \Bigl( \bigl|\svdf_1\A_{21}^1 \bigr| - \bigl|\svdf_2\A_{22}^2\bigr| \Bigr)^2 \\
					&\quad - \frac{1}{(1+\svdf_1^2)(1+\svdf_2^2)}\Bigl\{ (\svdf_1^2+\svdf_2^2)(\sec_M-\sigma) + 2 \svdf_1^2\svdf_2^2 (\sigma-\sec_N) \\
						& \qquad \qquad \qquad \qquad \qquad + \sigma(\svdf_1^2+\svdf_2^2-2\svdf_1^2\svdf_2^2) \Bigr\} \,.
			\end{align*}
			The assumption $\tr(\sind)\ge 0$ is equivalent to $\svdf_1^2\svdf_2^2\le 1$,
			which in turn implies the estimate
			\[
				\svdf_1^2+\svdf_2^2-2\svdf_1^2\svdf_2^2 \ge \svdf_1^2+\svdf_2^2 - 2 |\svdf_1\svdf_2| = \bigl( |\svdf_1|-|\svdf_2| \bigr)^2 \ge 0
			\]
			for the third curvature term.
			Therefore, all terms in the above differential equation are non-negative.
			By the maximum principle, we conclude that $\ln v$ must be constant on $M$. 
			Evaluating $0=\Delta\ln v$, we see that $\A=0$ everywhere.
			The estimate for the curvature term yields $\svdf_1^2=\svdf_2^2$,
			and then $\svdf_1^2=\svdf_2^2=0$ or $\svdf_1^2=\svdf_2^2=1$ everywhere.
			In the latter case, the other curvature terms imply $\sec_M=\sigma$ and
			$\sec_N=\sigma$ on $\dd f(TM)$.
		\end{enumerate}
	\end{remark}
	
	\begin{remark}[Weakly length-decreasing maps]
		\label{rem:WeaklyLengthDecreasing}
		By examining lemma \ref{lem:StrictLengthDecr} and the proof
		of theorem \ref{thm:ThmA}, 
		we note that if $\svdf_0^2\le 1$
		one does not need to impose a condition on 
		the second fundamental form.
		Thus, in the case at hand, we obtain an alternative proof of \cite[Theorem A]{SHS13}.
	\end{remark}
	
	Let us remark that there exists an abundance 
	of weakly length-decreasing minimal maps 
	which are not totally geodesic. In particular,
	these maps satisfy condition \eqref{eq:trCon}.

	\begin{example}[Holomorphic Maps, {\cite[Example 2(a)]{SHS13}}]
		For complex manifolds $M$ and $N$, the graph of a holomorphic map $f:M\to N$ is
		automatically minimal. By a result due to Ahlfors \cite{Ahlfors38} and its extension by Yau \cite{Yau78},
		for every holomorphic map $f:M\to N$ between a complete K\"{a}hler manifold
		$M$ with Ricci curvature bounded from below by a negative constant $-a$ and $N$
		a Hermitian manifold with holomorphic bisectional curvature bounded
		from above by a negative constant $-b$, it is $f^*\gN\le\frac{a}{b}\gM$.
	\end{example}

\end{document}